\documentclass[11pt,a4paper,dvipsnames]{article}
\usepackage[utf8]{inputenc}

\title{Convex Hulls of Curves: Volumes and Signatures}
\author{Carlos Am\'endola, Darrick Lee, Chiara Meroni}
\date{}
\usepackage{a4wide}
\usepackage{amssymb}
\usepackage{amsmath}
\usepackage{amsthm}
\usepackage{mathtools}
\usepackage{xcolor}
\usepackage{tikz,tikz-3dplot,tikz-cd}
\usepackage[all]{xy}
\usepackage{pgfplots}
\usepackage[margin=1cm]{caption}
\usepackage{subcaption}
\pgfplotsset{width=7cm, compat=newest}
\usepackage[top=3cm,bottom=2cm,left=3cm,right=3cm,marginparwidth=1.75cm]{geometry}
\usepackage{graphicx}
\graphicspath{{Figures/}}
\usepackage{xcolor}
\usepackage{tikz,tikz-3dplot,tikz-cd}

\definecolor{mygreen}{RGB}{37,170,37}
\definecolor{myteal}{RGB}{6,144,179}
\definecolor{myblue}{RGB}{59,100,189}
\definecolor{myorange}{RGB}{240,153,85}
\definecolor{myred}{RGB}{190,15,15}
\definecolor{myyellow}{RGB}{255,220,130}
\definecolor{mypurple}{RGB}{139,0,139}

\usetikzlibrary{decorations.pathreplacing,decorations.markings}
\tikzset{%
  symbol/.style={
    draw=none,
    every to/.append style={
      edge node={node [sloped, allow upside down, auto=false]{$#1$}}
    },
  },
  on each segment/.style={
    decorate,
    decoration={
      show path construction,
      moveto code={},
      lineto code={
        \path [#1]
        (\tikzinputsegmentfirst) -- (\tikzinputsegmentlast);
      },
      curveto code={
        \path [#1] (\tikzinputsegmentfirst)
        .. controls
        (\tikzinputsegmentsupporta) and (\tikzinputsegmentsupportb)
        ..
        (\tikzinputsegmentlast);
      },
      closepath code={
        \path [#1]
        (\tikzinputsegmentfirst) -- (\tikzinputsegmentlast);
      },
    },
  },
  mid arrow/.style={postaction={decorate,decoration={
        markings,
        mark=at position .5 with {\arrow[#1]{stealth}}
      }}},
}

\usepackage{stmaryrd}
\usepackage[hypertexnames=false,pdftex,backref=page,
pdfpagemode=UseNone,
breaklinks=true,
extension=pdf,
colorlinks=true,
linkcolor=myblue,
citecolor=myblue,
urlcolor=myblue]{hyperref}

\usepackage{etoolbox}
\patchcmd{\thebibliography}
  {\settowidth}
  {\setlength{\itemsep}{0pt plus 0.1pt}\settowidth}
  {}{}
\apptocmd{\thebibliography}
  {\small}
  {}{}

\newcommand{\R}{\mathbb{R}}

\newcommand{\bx}{\mathbf{x}}

\DeclareMathOperator{\Alt}{Alt} 
\DeclareMathOperator{\sgn}{sgn}
\newcommand{\cO}{\mathcal{O}}

\newcommand{\bv}{\mathbf{v}}

\DeclareMathOperator{\SO}{SO}
\DeclareMathOperator{\ext}{Ext}
\newcommand{\dd}{\mathrm{d}}

\newcommand{\obx}{\overline{\bx}}
\newcommand{\fT}{\mathfrak{T}}
\newcommand{\fD}{\mathfrak{D}}

\newcommand{\TorPath}{\mathsf{T}} % totally positive torsion
\newcommand{\SDetPath}{\mathsf{SD}} % positive determinant

\newcommand{\Cyc}{\mathsf{Cyc}} % cyclic paths
\newcommand{\Ord}{\mathsf{Ord}}
\newcommand{\Lip}{\text{Lip}}

\DeclareMathOperator{\conv}{conv}
\DeclareMathOperator{\vol}{vol}

\mathtoolsset{showonlyrefs,showmanualtags}

\newtheorem{theorem}{Theorem}[section]
\newtheorem{corollary}[theorem]{Corollary}
\newtheorem{lemma}[theorem]{Lemma}

\newtheorem{conjecture}[theorem]{Conjecture}

\theoremstyle{definition}
\newtheorem{definition}[theorem]{Definition}

\newtheorem{examplex}[theorem]{Example}
\newenvironment{example}
{\pushQED{\qed}\examplex}
{\popQED\endexamplex}

\newtheorem{remarkx}[theorem]{Remark}
\newenvironment{remark}
{\pushQED{\qed}\remarkx}
{\popQED\endremarkx}

\begin{document}
\maketitle

\begin{abstract}
Taking the convex hull of a curve is a natural construction in computational geometry. On the other hand, path signatures, central in stochastic analysis,  capture geometric properties of curves, although their exact interpretation for levels larger than two is not well understood. In this paper, we study the use of path signatures to compute the volume of the convex hull of a curve. We present sufficient conditions for a curve so that the volume of its convex hull can be computed by such formulae. The canonical example is the classical moment curve, and our class of curves, which we call \emph{cyclic}, includes other known classes such as $d$-order curves and curves with totally positive torsion. We also conjecture a necessary and sufficient condition on curves for the signature volume formula to hold. Finally, we give a concrete geometric interpretation of the volume formula in terms of lengths and signed areas.
\end{abstract}

\noindent

\section{Introduction}

Taking the convex hull of a curve is a classical geometric construction. Understanding both its computation and properties is important for non-linear computational geometry, with the case of space curves particularly relevant for applications such as geometric modeling \cite{sedykh1986structure, seong2004convex, ranestad2012convex}. 
On the other hand, volumes are a fundamental geometric invariant. Computing volumes of convex hulls leads to interesting isoperimetric problems in optimization \cite{melzak1960isoperimetric}, and has applications in areas like ecology \cite{cornwell2006trait} and spectral imaging \cite{messinger2010spectral}.

In the recent article \cite{PIM:ConvexHullPositiveTorsion}, the authors show that for curves which satisfy a \emph{totally positive torsion} property, one can compute the volume of their convex hull using a certain integral formula. The totally positive torsion property is a local convexity condition, and an example of such a path is the \emph{moment curve}
\[
    \bx(t) = (t, t^2, \ldots, t^d) : [0,1] \rightarrow \R^d.
\]
From the perspective of discrete geometry, the moment curve is also a canonical example of a larger class of curves called \emph{d-order curves}~\cite{Sturmfels:CyclicPolytopesdCurves}.
Our contribution is to extend the integral formulae of~\cite{PIM:ConvexHullPositiveTorsion} for volumes of convex hulls to the class of \emph{cyclic curves}, which are uniform limits of $d$-order curves. The motivation behind this name arises from the fact that $d$-order curves can be approximated by cyclic polytopes, which play a central role in our proof of this generalization.

Our method uses the notion of the \emph{path signature}~\cite{chen_integration_1958}, a powerful tool which is widely used in both stochastic analysis~\cite{friz_course_2020} and machine learning~\cite{lyons_signature_2022}, which represents a path as an infinite sequence of tensors. While it is well-known that the path signature characterizes paths up to \emph{tree-like equivalence}~\cite{hambly_uniqueness_2010}, individual terms of the path signature are difficult to interpret geometrically. The article~\cite{DR:InvariantsMultidimensional} shows that certain orthogonal invariants of the path signature can be understood as a notion of signed volume. In particular, they show that this orthogonal invariant computes the volume of the convex hull in the specific case of the moment curve. This result inspires our extension to the setting of cyclic curves, which forms a connection between the convex hull formulae of~\cite{PIM:ConvexHullPositiveTorsion} and the path signature.

\section{Classes of curves}

In this section, we consider several classes of paths for which the convex hull of the curve is well-behaved. Throughout this article, we consider Lipschitz-continuous paths $\bx = (x_1, \ldots, x_d) \in \Lip([0,1], \R^d)$. 
The totally positive torsion property \cite{PIM:ConvexHullPositiveTorsion} of a path is defined as follows.
\begin{definition}
    [Totally positive torsion ($\TorPath^d$)]\label{def:totally_pos_torsion}
    A path $\bx \in C^d((0,1), \R^d) \cap C([0,1],\R^d)$ has \emph{totally positive torsion} if all the leading principal minors of the matrix of derivatives
    \begin{equation}
        \mathfrak{T}(\bx)(t) \coloneqq (\bx'(t), \ldots, \bx^{(d)}(t))
    \end{equation}
    are positive for all $t \in (0,1)$. The space of totally positive torsion paths is denoted by $\TorPath^d$.
\end{definition}
Notice that this definition is dependent on the parametrization and is not invariant under the action of the orthogonal group. The space $\TorPath^d$ is therefore the space of all curves in $\R^d$ that admit a rotation which turns them into totally positive torsion paths. Also, this definition requires a high regularity of $\bx$, which should be at least $d$ times differentiable. In \cite{PIM:ConvexHullPositiveTorsion} we find also the following larger class of curves, used in many of their proofs.

\begin{definition}
    [Strictly positive determinant condition ($\SDetPath^d$)]\label{def:strict_det}
    A path $\bx \in \Lip([0,1], \R^d)$ satisfies the \emph{strictly positive determinant condition} if the property
    \begin{equation}
        \fD(\bx)(t_1, \ldots, t_d) \coloneqq \det(\bx'(t_1), \ldots, \bx'(t_d)) > 0, \quad 0 < t_1 < \ldots < t_d < 1
    \end{equation}
    holds where $\bx$ is differentiable. The space of such strictly positive determinant paths is denoted $\SDetPath^d$.
\end{definition}

On the other hand, the following class of curves is purely geometric and therefore independent of the parametrization.  

\begin{definition}
    [$d$-order path ($\Ord^d$)]\label{def:d_order}
    A path $\bx \in \Lip([0,1], \R^d)$ is a \emph{$d$-order path} if any affine hyperplane in $\R^d$ intersects the image of $\bx$ in at most $d$ points. This occurs exactly when the property
    \begin{equation}\label{eq:det_dorder}
        \det 
        \begin{pmatrix}
            1 & 1 & \dots & 1 \\
            \bx(t_0) & \bx(t_1) & \dots & \bx(t_d)
        \end{pmatrix}
        > 0
        \qquad \hbox{ for all } 0\leq t_0 < \ldots < t_d\leq 1
    \end{equation}
    holds. The space of $d$-order paths is denoted $\Ord^d$. 
\end{definition}

These curves have special properties and they have been studied in the literature intensively, although under many different names. We follow here the language of \cite{Sturmfels:CyclicPolytopesdCurves}. These curves were called \emph{comonotone} by Motzkin \cite{Motzkin:Comonotone,Motzkin:ConvexTypeVarieties}, who attributes the origin of the definition to Juel \cite{Juel:NonanalyticCurves}. This concept of monotonicity goes also back to Hjelmslev \cite{Hjelmslev:Monotones}. In Labourie \cite{Labourie:Hyperconvex} they are called \emph{hyperconvex}. Maybe the most famous name that $d$-order curves have is \emph{convex}. However, various authors give slightly different definitions of convexity, sometimes counting the intersection points of a curve with a hyperplane with multiplicities \cite{SedSha:ConjecturesConvexCurves, ShaSha:Corrigendum}, sometimes allowing some but not all the determinants in \eqref{eq:det_dorder} to be zero \cite{Karp:MomentCurves}. For instance, in \cite{Karp:MomentCurves} our $d$-order curves are called \emph{strictly convex}. We point out that this definition is not the same as the one of strictly convex curves used by Barner \cite{Barner:ConvexCurves} and Fabricius-Bjerre \cite{FabBje:StrictlyConvex,FabBje:PolygonsOrdern}.

Independently of the name, many results are known and conjectures have been stated. For instance, in \cite{KarStu:TchebycheffSystems} the authors compute volume and Caratheodory number of convex curves, and in \cite{schoenberg1954isoperimetric} a volume formula for their convex hull appears, in the case of closed even dimensional such curves. One of the peculiarities of $d$-order curves that we will exploit is that the convex hull of every $n$-tuple of points on the curve is a \emph{cyclic polytope} \cite{Ziegler:LecturesPolytopes} having those $n$ points as vertices. \medskip

By \cite[Lemma 3.1]{PIM:ConvexHullPositiveTorsion}, the totally positive torsion property implies the strictly positive determinant condition, hence $\TorPath^d \subseteq \SDetPath^d$. Furthermore, the proof of~\cite[Corollary 2.5]{PIM:ConvexHullPositiveTorsion} shows that the strictly positive determinant condition for a path $\bx$ implies that $\bx$ is a $d$-order path, hence $\SDetPath^d \subseteq \Ord^d$. As the following two examples show, both of these inclusions are strict.

\begin{example}[$\TorPath^d \subsetneq \SDetPath^d$]\label{ex:SDexample}
    Consider the path defined by
    \[
        \bx(t) = \big( t, (t-0.5)^4 \big).
    \]
    The determinant from the $\SDetPath^d$ condition is
    \[
        \fD(\bx)(t_1, t_2) = \begin{vmatrix} 1 & 1 \\ 4(t_1 - 0.5)^3 & 4(t_2 - 0.5)^3 \end{vmatrix} = 4(t_2 - 0.5)^3 - 4(t_1 - 0.5)^3 > 0
    \]
    whenever $t_1 < t_2$. However, the torsion matrix is
    \[
        \fT(\bx)(t) = \begin{pmatrix} 1 & 0 \\ 4(t-0.5)^3 & 12(t-0.5)^2 \end{pmatrix},
    \]
    and the determinant is zero when $t=0.5$. 
\end{example}

\begin{example}[$\SDetPath^d \subsetneq \Ord^d$]\label{ex:circle}
    Consider the unit circle path defined by 
    \[
        \bx(t) = \big(\cos(2\pi t), \sin(2\pi t)\big).
    \]
    It is clear that $\bx$ is a $2$-order path since any line in $\R^2$ can intersect this circle in at most $2$ points. However, the determinant
    \[
        \fD(\bx)\left(\frac{1}{4}, \frac{3}{4}\right) = \begin{vmatrix} -2\pi & 2\pi \\ 0 & 0 \end{vmatrix} = 0
    \]
    does not satisy the strictly positive determinant condition.
\end{example}

On one end, we can interpret the condition for $\TorPath^d$ to be a \emph{local} condition, in the sense that it is checked simply at individual points on the curve. On the other end, the condition for $\Ord^d$ is a \emph{global} condition, in the sense that it simultaneously takes $d+1$ points of the curve into consideration. 

As Example \ref{ex:circle} shows, the strictly positive determinant condition is already quite restrictive; in particular, any closed curve will violate this condition in a similar way. We find that the more general global condition of $\Ord^d$ is the correct notion for our purposes. In fact, we can further extend the class of $\Ord^d$; the name suggests a connection to \emph{cyclic} polytopes.

\begin{definition}
    [Cyclic paths ($\Cyc^d$)]\label{def:cyclic}
    A path $\bx \in \Lip([0,1], \R^d)$ is \emph{cyclic} if it is a limit, in the Lipschitz (or equivalently, uniform) topology, of $d$-order curves. The space of cyclic paths is denoted by $\Cyc^d$. 
\end{definition}
In particular, cyclic curves include some piecewise linear curves, which are not contained in any of the previous classes. It also allows subsets of the curve to lie in a lower dimensional space. See Figure \ref{fig:cyc_not_ord} for an example.  We point out that for a curve $\bx$ that spans the whole $\R^d$, cyclicity is the same as the relaxed version of condition \eqref{eq:det_dorder}, where the determinants are required to be non-negative. In the case that $\bx$ is contained in some lower dimensional subspace, the relaxed version of \eqref{eq:det_dorder} puts no condition on the curve, whereas cyclicity does.

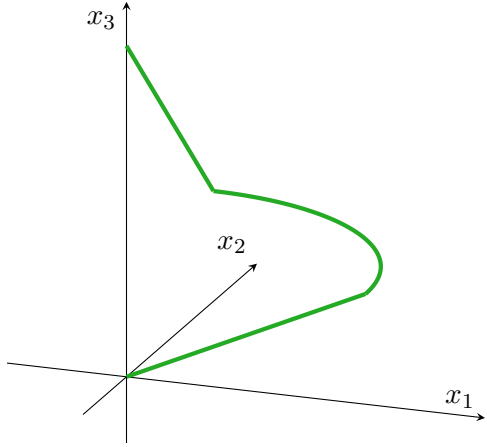
\begin{figure}
    \centering
  \begin{tikzpicture}
    \begin{axis}[view={20}{20},
    width=4in,
    height=4in,
    axis lines=center,
    xlabel=$x_1$,ylabel=$x_2$,zlabel=$x_3$,
    xticklabel=\empty,
    yticklabel=\empty,
    zticklabel=\empty,
    ylabel style={above left},
    xlabel style={above left},
    zlabel style={below left},
    xtick={0},
	ytick={0},
	ztick={0},
    xmin=-0.5,xmax=1.5,ymin=-0.5,ymax=1.5,zmin=-0.3,zmax=1.7]
    \draw [color=mygreen,ultra thick]
    (0,0,0) -- (1,0,1/2);
    \draw [color=mygreen,ultra thick]
    (0,1,1/2) -- (0,0,1.5);
    \addplot3+[no markers, samples=200, samples y=0, domain=0:1/4, variable=\t, style=ultra thick, color = mygreen]
                              ({cos(2*pi*\t r)}, {sin(2*pi*\t r)}, {1/2});
    \end{axis}
    \end{tikzpicture}
    \caption{A curve in $\Cyc^d \setminus \Ord^d$, containing two line segments and an arc which lies in the plane $\{x_3=\frac{1}{2}\}$.}
    \label{fig:cyc_not_ord}
\end{figure}
In conclusion, we obtain the following sequence of strict inclusions, transitioning from local to global properties of curves:
\begin{equation}
    \TorPath^d \subset \SDetPath^d \subset \Ord^d \subset \Cyc^d.
\end{equation}

\begin{remark}
Cyclic curves are generalizations of \emph{convex curves}, which are connected to the Shapiro-Shapiro conjecture, stated in the early 1990s, regarding real Schubert calculus and how to obtain totally real configurations. 
For the state-of-the-art on this conjecture, see \cite{ShaSha:Corrigendum}.
\end{remark}

\section{Path signatures and volume}

In this section, we introduce the connection between volume formulae for convex hulls of curves and the \emph{path signature}, a characterization of curves based on iterated integrals~\cite{chen_integration_1958}. In particular, we show that the volume of the convex hull of a cyclic curve can be written in terms of an antisymmetrization of the path signature. This connection was first considered in~\cite{DR:InvariantsMultidimensional} in the case of the moment curve. 

\begin{definition}
    Suppose $\bx = (x_1, \ldots, x_d) \in \Lip([0,1], \R^d)$. The \emph{level $k$ path signature} of $\bx$ is a tensor $\sigma^{(k)}(\bx) \in (\R^d)^{\otimes k}$, defined by
    \begin{align}
        \sigma^{(k)}(\bx) \,\,\coloneqq\,\, \int_{\Delta^k} \bx'(t_1) \otimes \cdots \otimes \bx'(t_k) \, \dd t_1 \cdots \dd t_k \in (\R^d)^{\otimes k}
    \end{align}
    where the integration is over the $k$-simplex
    \begin{align*}
        \Delta^k \,\,\coloneqq \,\,\{ 0 \leq t_1 < \cdots < t_k \leq 1\}.
    \end{align*}
    Given a multi-index $I = (i_1, \ldots, i_k) \in [d]^k$, the \emph{path signature of $\bx$ with respect to $I$} is
    \begin{align}
        \sigma_I(\bx) \,\,\coloneqq\,\, \int_{\Delta^k} x'_{i_1}(t_1) \cdots x'_{i_k}(t_k)  \,\dd t_1 \cdots \dd t_k \in \R.
    \end{align}
\end{definition}

The \emph{path signature} $\sigma(\bx)$ of a path $\bx$ is the formal power series obtained by summing up all levels $\sigma^{(k)}(\bx)$. It is well known that the path signature characterizes paths up to \emph{tree-like equivalence}~\cite{chen_integration_1958, hambly_uniqueness_2010}; however, the individual entries of the signature $\sigma_I(\bx)$ are often difficult to understand geometrically. We aim to provide a geometric interpretation of signature terms via antisymmetrization into the exterior algebra $\Lambda(\R^d) \coloneqq \bigoplus_{k=0}^d \Lambda^k \R^d$, where $\Lambda^k \R^d$ is the vector space of alternating tensors of $\R^d$ of degree $k$. These tensors are indexed using order preserving injections $P: [k] \rightarrow [d]$, denoted here by $\cO_{k,d}$.

\begin{definition}
    Suppose $\bx = (x_1, \ldots, x_d) \in \Lip([0,1], \R^d)$. The \emph{level $k$ alternating signature} of $\bx$ is a tensor $\alpha^{(k)}(\bx) \in \Lambda^k\R^d$, defined by
    \begin{align}
        \alpha^{(k)}(\bx) \,\,\coloneqq\,\, \int_{\Delta^k} \bx'(t_1) \wedge \cdots \wedge \bx'(t_k) \, \dd t_1 \cdots \dd t_k \in \Lambda^k\R^d.
    \end{align}
    Given an order-preserving injection $P: [k] \rightarrow [d]$, the \emph{alternating signature of $\bx$ with respect to $P$} is
    \begin{align}
        \alpha_P(\bx) \,\,\coloneqq\,\, \frac{1}{k!}\int_{\Delta^k} \det(x'_{P(1)}(t_1), \ldots, x'_{P(k)}(t_k))  \,\dd t_1 \cdots \dd t_k \in \R.
    \end{align}
\end{definition}

The alternating signature can equivalently be defined using the \emph{antisymmetrization}
\begin{align}
    \Alt: (\R^d)^{\otimes k} \rightarrow \Lambda^k (\R^d), \quad \quad \Alt(\sigma)_P \,\,\coloneqq \,\,\frac{1}{k!} \sum_{\tau \in \Sigma_d} \sgn(\tau) \sigma_{P \circ \tau},
\end{align}
where
\begin{align}
    \alpha^{(k)}(\bx) \,\,\coloneqq \,\,\Alt(\sigma^{(k)}(\bx)) \in \Lambda^k(\R^d).
\end{align}

We will primarily be interested in the level $k=2$ and the level $k=d$ alternating signature. Let $\bx \in \Lip([0,1],\R^d)$ and consider the alternating signature at level $k=2$. At this level, $\alpha^{(2)}(\bx)$ can be viewed as a $d \times d$ antisymmetric matrix whose $(i,j)$-th entry is
\begin{align*}
    \alpha_{i,j}(\bx) \,\,=\,\, \frac{1}{2} \left(\sigma_{i,j}(\bx) - \sigma_{j,i}(\bx) \right),
\end{align*}
which is exactly the \emph{signed area} (see Figure \ref{fig:signed_area}) of the path $\bx$ projected to the $(e_i, e_j)$-plane. 
   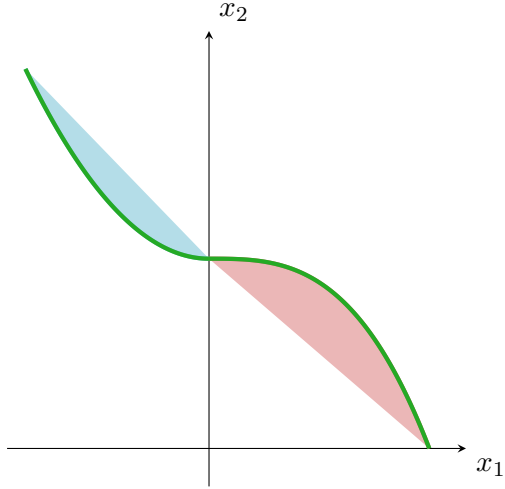
\begin{figure}
    \centering
  \begin{tikzpicture}
      \begin{axis}[view={65}{50},
    width=3in,
    height=3in,
    axis lines=center,
    xlabel=$x_1$,ylabel=$x_2$,
    xticklabel=\empty,
    yticklabel=\empty,
    ylabel style={above right},
    xlabel style={below right},
    xtick={0},
	ytick={0},
    xmin=-1.1,xmax=1.4,ymin=-0.2,ymax=2.2]
    \addplot+[no markers, fill, samples=200, samples y=0, domain=0:6/5, variable=\x, style=ultra thick, color = myred!60, fill , fill opacity = 0.50]
                              ({\x}, {(-5/6*\x)^3+1});
    \addplot+[no markers, fill, samples=200, samples y=0, domain=-1:0, variable=\x, style=ultra thick, color = myteal!60, fill , fill opacity = 0.50]
                              ({\x}, {(\x)^2+1});
    \addplot+[no markers, samples=200, samples y=0, domain=0:6/5, variable=\x, style=ultra thick, color = mygreen]
                              ({\x}, {(-5/6*\x)^3+1});
    \addplot+[no markers, samples=200, samples y=0, domain=-1:0, variable=\x, style=ultra thick, color = mygreen]
                              ({\x}, {(\x)^2+1});
    \end{axis}
    \end{tikzpicture}
    \caption{The green curve $\bx$ whose signed area is the sum of the areas of the red and blue regions, where the red area comes with positive sign and the blue area comes with negative sign.}
    \label{fig:signed_area}
\end{figure}
In the case of level $k=d$, the exterior power $\Lambda^d \R^d$ is one-dimensional, and we can express the level $d$ alternating signature as
\begin{align*}
    \alpha^{(d)}(\bx) = \frac{1}{d!}\int_{\Delta^d} \det(\bx'(t_1), \ldots, \bx'(t_d)) \mathrm{d}t_1\ldots \mathrm{d}t_d.
\end{align*}
Furthermore, the top-level alternating signature is rotation invariant. 
\begin{theorem}[\cite{DR:InvariantsMultidimensional}]
\label{thm:so_invariance}
    The level $d$ alternating signature for paths $\bx \in \Lip([0,1],\R^d)$ is invariant under the special orthogonal group $\SO(d)$ (where the action acts pointwise over $[0,1]$), i.e., given $V \in \SO(d)$,
    \begin{align*}
        \alpha^{(d)}(V\bx) \,\,=\,\, \alpha^{(d)}(\bx). 
    \end{align*}
\end{theorem}

\subsection{Convex hull formulae for cyclic curves}

In~\cite{DR:InvariantsMultidimensional}, the authors interpret the alternating signature $\alpha^{(d)}$ as the \emph{signed-volume} of a curve, and furthermore show that $\alpha^{(d)}$ is the volume of the convex-hull of the moment curve. The first part of our generalization extends this to $d$-order curves. This was initially proved using other methods in \cite[Theorem 6.1]{KarStu:TchebycheffSystems}.

\begin{theorem}\label{thm:dordervolformula}
    Let $\bx : [0,1] \rightarrow \R^d$ be a $d$-order curve. Then
    \[
        \vol(\conv(\bx)) = \alpha^{(d)}(\bx).
    \]
\end{theorem}
\begin{proof}
    Consider the $n+1$ points $p_0 = \bx(t_0), \dots, p_n = \bx(t_n)$ on the $d$-order curve, with $0\leq t_0 < \ldots < t_n\leq 1$. Denote by $C_n$ their convex hull $\conv(p_0,\ldots,p_n)$, which is a cyclic polytope. We can realize a triangulation of $C_n$ by pulling one vertex, as in \cite[Lemma 3.29]{DR:InvariantsMultidimensional}. Indeed, by Gale evenness criterion, a triangulation of a cyclic polytope is given by simplices with vertices $p_{i_0}, \ldots , p_{i_d}$ satisfying
    \begin{itemize}
        \item for even $d$: $i_0 = 0$ and $i_{2\ell +1} + 1 = i_{2\ell + 2}$ for any $1\leq \ell \leq d$;
        \item for odd $d$: $i_0 = 0$, $i_d = n$, and $i_{2\ell +1} + 1 = i_{2\ell + 2}$ for any $1\leq \ell \leq d-1$.
    \end{itemize}
    Let us denote by $\mathcal{I}$ the set of all $d$-tuples of indices that satisfy these conditions.
    Therefore, the volume of $C_n$ is the sum of the volumes of all these simplices. Because $\bx$ is a $d$-order curve
    \begin{equation}
        \det 
        \begin{pmatrix}
            1 & 1 & \dots & 1 \\
            \bx(s_0) & \bx(s_1) & \dots & \bx(s_d)
        \end{pmatrix}
        > 0
    \end{equation}
    for every choice of increasing $0 \leq s_0 < \ldots < s_d\leq 1$. Hence, 
    \begin{equation}\label{eq:volume_cyclic_polytope}
        \vol (C_n) = \frac{1}{d!} \sum_{\{i_0,\ldots,i_d\}\in \mathcal{I}} \det 
        \begin{pmatrix}
            1 & 1 & \dots & 1 \\
            p_{i_0} & p_{i_1} & \dots & p_{i_d}
        \end{pmatrix}.
    \end{equation}
    Since $\bx$ has bounded variation, we can take the limit of \eqref{eq:volume_cyclic_polytope} for $n\to \infty$, as in \cite[Lemma 3.29]{DR:InvariantsMultidimensional}. In particular, we use the continuity of convex hulls from~\cite[Section 1.8]{Schneider:BrunnMinkowskiTheory} and the continuity of truncated signatures from~\cite[Proposition 7.63]{friz_multidimensional_2010}. This gives the formula $\vol(\conv(\bx)) = \alpha^{(d)}(\bx)$.
\end{proof}

However, since the volume and the path signature are both continuous, the volume formula will still hold for limits of $d$-order curves
\begin{theorem}\label{thm:cyclicvolformula}
    Let $\bx : [0,1] \rightarrow \R^d$ be a cyclic curve. Then,
    \[
        \vol(\conv(\bx)) = \alpha^{(d)}(\bx).
    \]
\end{theorem}
\begin{proof}
Since $\bx$ is cyclic, we can write it as a limit of $d$-order curves $\bx_k$ in $\R^d$, so that $\bx = \lim_{k\rightarrow \infty} \bx_k$ . Since $\conv$ is a continuous operation
\cite[Section 1.8]{Schneider:BrunnMinkowskiTheory}, it holds that $\conv (\bx) = \lim \conv (\bx_k)$. By Theorem \ref{thm:dordervolformula}, we know that $\vol( \conv (\bx_k)) = \alpha^{(d)}(\bx_k) $. Hence we have
\begin{equation}
    \vol( \conv (\bx)) = \lim_{k\to \infty} \vol( \conv (\bx_k)) = \lim_{k\to \infty} \alpha^{(d)}(\bx_k) = \alpha^{(d)}(\bx) .
\end{equation}
where the first equality follows by the continuity of the volume operation \cite[Theorem 1.8.20]{Schneider:BrunnMinkowskiTheory}; the second equality is due to Theorem \ref{thm:dordervolformula}; the last equality follows from the stability of the signature inside the class of continuous curves of bounded variation~\cite[Proposition 7.63]{friz_multidimensional_2010}.
\end{proof}

Starting from a cyclic curve there is a natural way to construct a centrally symmetric convex body  that has the same volume of the convex hull of the curve itself. Indeed, let $\bx\subset \R^d$ be a cyclic curve and interpret the vector $\bx'$ of first derivatives of the parametrization of $\bx$ as a random vector of $\R^d$. Then one can define the \emph{zonoid} $\mathcal{Z}(\bx')$ associated to the given random vector, as in \cite{Vitale:RandomDeterminantsZonoids}. We have that
\begin{align}
    \vol( \conv (\bx)) &= \frac{1}{d!} \int_{\Delta^d} \det(\bx'(t_1), \ldots, \bx'(t_d)) \, \mathrm{d}t_1\ldots \mathrm{d} t_d \\
    &= \frac{1}{d!\cdot d!} \int_{[0,1]^d} | \det(\bx'(t_1), \ldots, \bx'(t_d)) | \, \mathrm{d}t_1\ldots \mathrm{d} t_d = \frac{1}{d!} \vol (\mathcal{Z}(\bx')).
\end{align}
We illustrate this with the following example.

\begin{example}
    Let $\bx : [0,1] \to \R^2$ be the moment curve parametrized by $(t,t^2)$. Then we can compute explicitly the zonoid $\mathcal{Z}(\bx')$. Its support function \cite{Schneider:BrunnMinkowskiTheory} is by definition
    \begin{equation}
        h_\mathcal{Z} (u,v) = \frac{1}{2} \int_0^1 | \langle (u,v) , \bx'(t) \rangle | \, \mathrm{d} t = \frac{1}{2} \int_0^1 | u + 2 v t | \, \mathrm{d} t
    \end{equation}
    for $(u,v)\in \R^2$. Via a case study one can solve the above integral and find the equations of the dual body $\mathcal{Z}(\bx')^\circ$. For more on zonoids and convex bodies see \cite{Schneider:BrunnMinkowskiTheory}. This allows to compute the equations of the semialgebraic zonoid 
   \begin{equation}
       \mathcal{Z}(\bx') = \{ (x,y)\in \R^2 \,|\, -4 x^2 + 4 x - 4 y + 1 \geq 0, -4 x^2 - 4 x + 4 y + 1 \geq 0 \}
   \end{equation}
   displayed in Figure \ref{fig:moment2d}.
   \begin{figure}
    \centering
    \begin{tikzpicture}
      \begin{axis}[view={65}{50},
    width=3in,
    height=3in,
    axis lines=center,
    xlabel=$x_1$,ylabel=$x_2$,
    xticklabel=\empty,
    yticklabel=\empty,
    ylabel style={above right},
    xlabel style={below right},
    xtick={0},
	ytick={0},
    xmin=-0.2,xmax=1.2,ymin=-0.2,ymax=1.2]
    \addplot+[no markers, fill, samples=200, samples y=0, domain=0:1, variable=\x, style=ultra thick, color = mygreen, fill , fill opacity = 0.50]
                              ({\x}, {\x^2});
    \draw [color=mygreen,ultra thick]
    (0,0) -- (1,1);
    \end{axis}
    \end{tikzpicture}
  \begin{tikzpicture}
      \begin{axis}[view={65}{50},
    width=3in,
    height=3in,
    axis lines=center,
    xlabel=$x_1$,ylabel=$x_2$,
    xticklabel=\empty,
    yticklabel=\empty,
    ylabel style={above right},
    xlabel style={below right},
    xtick={0},
	ytick={0},
    xmin=-0.7,xmax=0.7,ymin=-0.7,ymax=0.7]
    \addplot+[no markers, fill, samples=200, samples y=0, domain=-0.5001:0.5002, variable=\x, style=ultra thick, color = mygreen, fill , fill opacity = 0.50]
                              ({\x}, {-\x^2 + \x + 1/4});
    \addplot+[no markers, color = mygreen, fill, fill opacity = 0.50, samples=200, samples y=0, domain=-0.5:0.5, variable=\x, style=ultra thick]
                              ({\x}, {\x^2 + \x - 1/4});
    \end{axis}
    \end{tikzpicture}
    \caption{Left: the convex hull of the moment curve $\bx(t) = (t,t^2)$ for $t\in [0,1]$. Right: the zonoid $\mathcal{Z}(\bx')$.}
    \label{fig:moment2d}
\end{figure}
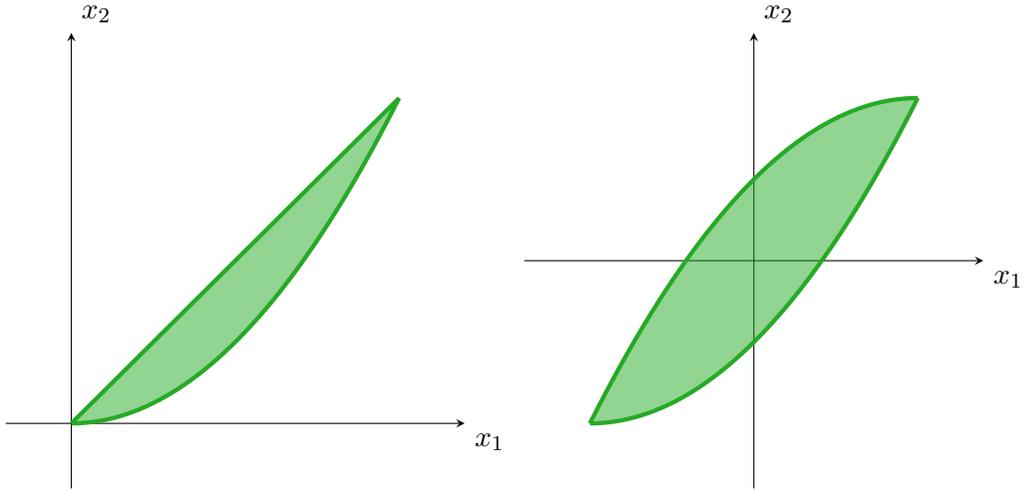
   A computation shows that $\vol ( \mathcal{Z}(\bx') ) = \frac{1}{3} = 2 \cdot \vol ( \conv (\bx))$.
\end{example}

\begin{remark}
    Let $\bx \subset \R^d$ be a piecewise linear cyclic curve parametrized by
    \begin{equation}
        \bx(t) = (a^{(i)}_1 t + b^{(i)}_1, \ldots , a^{(i)}_d t + b^{(i)}_d) \qquad \hbox{ for } t\in [m_i,M_i]
    \end{equation}
    for $N$ vectors $a^{(i)},b^{(i)} \in \R^d$, and $0 = m_1 < \ldots < M_{i-1} = m_i < \ldots < M_N = 1$ in $[0,1]$. 
    Then, $\mathcal{Z}(\bx')$ has support function $
        h_{\mathcal{Z}(\bx')}(u) = \sum_{i=1}^N \frac{M_i - m_i}{2} |\langle u, a^{(i)} \rangle|$.
    Therefore, $\mathcal{Z}(\bx')$ is the zonotope
    $$
    \mathcal{Z}(\bx') = \sum_{i=1}^N \left[ -\frac{(M_i - m_i)}{2} a^{(i)} , \frac{(M_i - m_i)}{2} a^{(i)} \right].
    $$
\end{remark}

\subsection{The logarithmic curve}
An example of a curve for which we can use the volume formula is the \emph{logarithmic curve}.
Fix $d$ distinct non-negative real numbers $a_1, a_2, \ldots , a_d$ and consider the associated logarithmic curve, or in short log-curve, parameterized by 
\begin{align}\label{eq:log-curve}
\bx : [0,1] &\to \R^d, \quad
t \, \mapsto 
\begin{pmatrix}
\log(1+a_1 t)\\
\log(1+a_2 t)\\
\vdots \\
\log(1+a_d t)
\end{pmatrix}.
\end{align}
We prove that log-curves are $d$-order, and therefore they are cyclic. 
\begin{lemma}
    The log-curve in \eqref{eq:log-curve} is a $d$-order curve.
\end{lemma}
\begin{proof}
Consider $0\leq t_0<\ldots<t_d\leq 1$ and take the associated points $\bx(t_i)$ on the curve. We need to prove that the following $(d+1)\times (d+1)$ matrix
\begin{equation}\label{eq:matrix_cyclic}
\begin{bmatrix}
\bx(t_0) & \ldots & \bx(t_d) \\
1 & \ldots & 1
\end{bmatrix}
\end{equation}
is not singular.
Suppose by contradiction that it is singular. Hence, there exist $c_1,\ldots, c_d \in \R$ such that
\begin{equation}\label{eq:linear_combination_log}
c_1 \log(1+a_1 t_i) + \ldots + c_d \log(1+a_d t_i) \,\,=\,\, 1
\end{equation}
for all $i=0,\ldots,d$. Let us then study the function
\begin{equation}
f(t) \,\,=\,\, c_1 \log(1+a_1 t) + \ldots + c_d \log(1+a_d t) - 1.
\end{equation}
Equation \eqref{eq:linear_combination_log} implies that $f$ has at least $d+1$ distinct roots in the interval $[0,1]$, where it is differentiable. However,
\begin{equation}
    f'(t) \,\,=\,\, \frac{g(t)}{(1+a_1 t)\cdot \ldots \cdot (1+a_d t)}
\end{equation}
where $g(t)$ is a polynomial of degree $d-1$.
Hence $f'$ has at most $d-1$ roots on $\R$. This implies that $f$ has at most $d$ roots in the interval $[0,1]$, which gives a contradiction, since $t_i\neq t_j$ for $i\neq j$. Therefore, no $d+1$ points of $\bx$ belong to any hyperplane, and hence the log-curve is a $d$-order curve.
\end{proof}

We thus have that $\bx \in \Cyc^d$, so we can compute the volume of its convex hull using the signature formula from Theorem \ref{thm:cyclicvolformula}. Explicitly, for $d=3$ and distinct non-negative parameters $a,b,c \neq 0$, we obtain
\begin{align}
    \vol (\conv (\bx)) &= \frac{1}{6} \int_0^1\! \int_0^w\! \int_0^v \frac{a b c}{(1 + a u) (1 + b v) (1 + c w)} - \frac{a b c}{(1 + b u) (1 + a v) (1 + c w)} \\
    & \qquad - \frac{a b c}{(1 + a u) (1 + c v) (1 + b w)} + \frac{a b c}{(1 + c u) (1 + a v) (1 + b w)} \\
    & \qquad + \frac{a b c}{(1 + b u) (1 + c v) (1 + a w)} - \frac{a b c}{(1 + c u) (1 + b v) (1 + a w)} \: \mathrm{d}u\, \mathrm{d}v\, \mathrm{d}w \\
    &= G(-\frac{1}{a}, -\frac{1}{b}, -\frac{1}{c} ; 1) - G(-\frac{1}{b}, -\frac{1}{a}, -\frac{1}{c} ; 1) + G(-\frac{1}{a}, -\frac{1}{c}, -\frac{1}{b} ; 1) \\
    & \qquad - G(-\frac{1}{c}, -\frac{1}{a}, -\frac{1}{b} ; 1) + G(-\frac{1}{b}, -\frac{1}{c}, -\frac{1}{a} ; 1) - G(-\frac{1}{c}, -\frac{1}{b}, -\frac{1}{a} ; 1),
\end{align}
where $G(z_1,\ldots,z_3;1)$ is a multiple polylogarithm, as defined in \cite[Section 8.1]{Weinzierl:FeynmanIntegrals}. Therefore, the volume of the convex hull of the log-curve $\bx$ is a combination of multiple polylogarithms. The same holds in higher dimensions: 
\begin{equation}
    \vol (\conv (\bx)) = \sum_{\tau \in \Sigma_d} \sgn(\tau) G\left(-\frac{1}{a_{\tau(1)}}, -\frac{1}{a_{\tau(2)}}, \ldots, -\frac{1}{a_{\tau(d)}} ; 1\right)
\end{equation}
where $\bx$ is a log-curve in $\R^d$ with parameters $a_1,\ldots,a_d$.

\subsection{Towards a necessary condition}

Since we have extended the class of curves for which the volume of their convex hulls equals the alternating signature, it is natural to wonder whether the sufficient condition of being cyclic is also a necessary condition for the volume formula to hold. In other words, we would like for the converse of Theorem~\ref{thm:cyclicvolformula} to hold. Unfortunately, this is not the case, as the following example shows.

\begin{example}
Consider the curve in Figure~\ref{fig:counterex}. 
\begin{figure}
    \centering
  \begin{tikzpicture}
  \path [draw=mygreen,ultra thick,postaction={on each segment={mid arrow=black}}]
  (0,0) -- (3,0) -- (3,2) -- (0,2) -- (0,0)
  (3,0) -- (5,0) -- (5,6) -- (4,6) -- (4,4) -- (1,4) -- (1,6) -- (0,6) -- (0,0);
\end{tikzpicture}
    \caption{Non-cyclic curve for which the convex hull volume formula holds.}
    \label{fig:counterex}
\end{figure}
Here the path $\bx$ starts (end ends) on the bottom left, traces the small cycle and then goes around the outer loop. Note that the convex hull $\conv(\bx)$ is a solid rectangle, whose area will equal the sum of the areas enclosed by each of the two loops. The latter is precisely the alternating signature, so
\[
        \vol(\conv(\bx)) = \alpha^{(2)}(\bx).
    \]
However, the curve $\bx$ is not cyclic because it does not have the following property.
\end{example}
\begin{lemma}
    Let $\bx\in\Cyc^d$. Then $\bx \subset \partial \conv(\bx)$.
\end{lemma}
\begin{proof}
    By definition, $\bx$ is a Lipschitz limit of $d$-order curves $\{\bx_k\}_{k=1}^\infty$, and therefore pointwise we have that $\bx(t) = \lim \bx_k(t)$. In particular this implies that, with respect to the Hausdorff metric, $\conv (\bx_k) \to \conv(\bx)$. Since every $d$-order curve is contained in the boundary of its convex hull, by the continuity of the limits also $\bx$ is contained in $\partial \conv(\bx)$.
\end{proof}

Notice that the condition $\bx \subset \partial \conv(\bx)$ is not enough for the volume formula to be true. Also strengthening this assumption to $\bx \subset \ext( \conv(\bx))$, where $\ext$ denotes the set of extreme points, does not work. Indeed, let $\bx$ be the curve parametrized by $[0,1]\ni t\mapsto (\cos 2\pi t,\sin 2\pi t, \cos 6\pi t)\in \R^3$.
Then, $\alpha^{(3)}(\bx) = 0$ whereas $\vol(\conv(\bx))\neq 0$.

Cyclic curves satisfy the property that any of their subpaths is also cyclic, and thus the volume formula holds for all subpaths of a cyclic curve. We conjecture that cyclic curves are the largest class of curves for which the signature formula holds for all subpaths. 
\begin{conjecture}\label{conj:subpaths}
    Let $\bx:[0,1]\to\R^d$ be a Lipschitz path. Then, $\bx$ is cyclic if and only if  $\vol(\conv(\widetilde{\bx}))=\alpha^{(d)}(\widetilde{\bx})$ holds for all restrictions $\widetilde{\bx}: [a,b] \rightarrow \R^d$ of $\bx$ to a subinterval $[a,b] \subset [0,1]$.
\end{conjecture}

\section{Lower-level decomposition}

The path signature is equipped with a shuffle algebra structure, and this induces a decomposition of the higher level alternating signature into first and second level terms. Furthermore, we can reinterpret the signature volume formula in terms of signed areas.
We begin with the decomposition of the alternating signature from~\cite{DR:InvariantsMultidimensional}.

\begin{lemma}[\cite{DR:InvariantsMultidimensional}, Lemma 3.17]
\label{thm:DR_decomposition}
    Let $\bx \in \Lip([0,1], \R^d)$ and suppose $k \leq d$.
    \begin{itemize}
        \item \textbf{($k$ odd).} Suppose $(P : [k] \rightarrow [d]) \in \cO_{k,d}$. For $i \in [k]$, let $P_i \in \cO_{k-1,d}$ be defined by
        \begin{align}
            P_i(r) \,\,=\,\, \begin{cases}
                    P(r) & \text{if } r < i, \\
                    P(r+1) &  \text{if } r \geq i.
                \end{cases}
        \end{align}
        Then
        \begin{align}
        \label{eq:alternating_decomp_odd}
            \alpha_P(\bx) \,\,=\,\, \frac{1}{k}\sum_{i=1}^k(-1)^{i+1} \sigma_{P(i)}(\bx) \cdot \alpha_{P_i}(\bx).
        \end{align}

        \item \textbf{($k$ even).} Suppose $(P: [k] \rightarrow [d]) \in \cO_{k,d}$. Then
        \begin{align}
        \label{eq:alternating_decomp_even}
            \alpha_P(\bx) \,\,=\,\, \frac{1}{k!(k/2)!}\sum_{\tau \in \Sigma_k} \sgn(\tau) \prod_{r=1}^{k/2} \alpha_{P(\tau(2r-1)), P(\tau(2r))}(\bx).
        \end{align}
    \end{itemize}
\end{lemma}

\begin{remark}
    The form of the result in~\cite[Lemma 3.17]{DR:InvariantsMultidimensional} and stated here is slightly different. In~\cite{DR:InvariantsMultidimensional}, results are stated in terms of shuffles of the indices, and the alternating signature is \emph{not} normalized by $\frac{1}{k!}$.
\end{remark}

Combining the two formulae in Lemma \ref{thm:DR_decomposition}, we can rewrite the odd-level alternating signature in terms of level $1$ and $2$ signature terms.

\begin{corollary}
\label{cor:odd_DR_decomposition}
    Let $\bx \in \Lip([0,1], \R^d)$, suppose $k \leq d$ is odd and $P \in \cO_{k,d}$. Then
    \begin{align}
    \label{eq:alternating_decomp_odd2}
        \alpha_P(\bx) \,\,=\,\, \frac{1}{k!\left(\frac{k-1}{2}\right)!} \sum_{\tau \in \Sigma_k} \sgn(\tau) \sigma_{P(\tau(1))}(\bx) \prod_{r=1}^{(k-1)/2} \alpha_{P_i(\tau(2r)), P_i(\tau(2r+1))}(\bx). 
    \end{align}
\end{corollary}
\begin{proof}
    By applying the odd formula followed by the even formula, we obtain
    \begin{align*}
        \alpha_P(\bx) \,\,=\,\, \frac{1}{k!\left(\frac{k-1}{2}\right)!} \sum_{i=1}^k \sum_{\tau \in \Sigma_{k-1}} (-1)^{i+1}\sgn(\tau) \sigma_{P(i)}(\bx) \prod_{r=1}^{(k-1)/2} \alpha_{P_i(\tau(2r-1)), P_i(\tau(2r))}(\bx). 
    \end{align*}
    Given some $i \in [k]$ and $\tau \in \Sigma_{k-1}$, we define a new permutation $\tau' \in \Sigma_k$ by
    \begin{align*}
        \tau'(r)  \,\,=\,\, \begin{cases} 
                i & \text{if } r = 1, \\
                \tau(r-1) & \text{if } 1 <r \leq k, \, \tau(r-1) < i, \\
                \tau(r-1)+1 & \text{if } 1 < r \leq k, \, \tau(r-1) \geq i.
\end{cases}
    \end{align*}
    Note that
    \begin{align*}
        \sgn(\tau') \,\,=\,\, (-1)^{i+1}\sgn(\tau)
    \end{align*}
    since we can write $\tau'$ as the composition of $\tau$ with a cycle of length $i$. We then have
    \begin{align}
        \alpha_P(\bx) \,\,=\,\, \frac{1}{k!\left(\frac{k-1}{2}\right)!} \sum_{\tau' \in \Sigma_k} \sgn(\tau') \sigma_{P(\tau'(1))}(\bx) \prod_{r=1}^{(k-1)/2} \alpha_{P_i(\tau'(2r)), P_i(\tau'(2r+1))}(\bx). 
    \end{align}
\end{proof}

\subsection{Volumes in terms of signed areas}
In this section, we discuss how the alternating signature volume formula can be interpreted as a product of signed areas (in the even case) times an additional displacement (in the odd case). This is done by decomposing the top degree alternating signature $\alpha^{(d)}(\bx)$ into level $1$ and $2$ components, as shown in Lemma~\ref{thm:DR_decomposition} and Corollary~\ref{cor:odd_DR_decomposition}. We begin by reducing the odd dimensional setting to the even dimensional setting.\medskip

Suppose $\bx = (x_1, \ldots, x_d): [0,1] \rightarrow \R^d$, where $d=2n+1$ and $\bx(1) \neq \bx(0)$. Without loss of generality (due to the $SO(d)$ invariance of the alternating signature), we suppose that $\bx(1) - \bx(0)$ is restricted to the $n+1$ coordinate. Then, the decomposition in Equation~\ref{eq:alternating_decomp_odd} is
\[
    \alpha^{(d)}(\bx) = \frac{1}{d!} (x_d(1) - x_d(0))\cdot\alpha^{(2n)}(\obx),
\]
where $\obx = (x_1, \ldots, x_{2n}) : [0,1] \rightarrow \R^{2n}$. Thus, it remains to interpret the top-level alternating signature of even-dimensional paths.

Now, suppose $\bx: [0,1] \rightarrow \R^d$, where $d=2n$ is even. Because $\alpha^{(2)}(\bx)$ is a antisymmetric matrix, it has purely imaginary eigenvalues which come in conjugate pairs. By applying a (complex) rotation between each of the conjugate pairs, we can block-diagonalize the matrix $\alpha^{(2)}$ as follows:
\begin{align}
\label{eq:block_diagonalize}
    \alpha^{(2)}(\bx) \,\,=\,\, Q \Lambda Q^T,
\end{align}
where
\begin{align}
\label{eq:block_diag}
    \Lambda \,\,=\,\, \begin{pmatrix} 0 &  \lambda_1 & 0 & 0 & \cdots & 0 & 0\\
    -\lambda_1 & 0 & 0 &0 & \cdots & 0 & 0\\
    0 & 0 & 0 & \lambda_2 & \cdots & 0 & 0 \\
    0 & 0 & -\lambda_2 & 0 & \cdots & 0 &0 \\
    \vdots & \vdots & \vdots & \vdots & &\vdots & \vdots \\
    0 &0& 0 &0 & \cdots & 0 & \lambda_n \\
    0 & 0 &0 &0 & \cdots & -\lambda_n & 0 \end{pmatrix}
\end{align}
and $Q$ is the orthogonal matrix of eigenvectors
\begin{align}
    Q \,\,=\,\, \begin{pmatrix} \bv_1 & \bv_2 & \cdots & \bv_{2n-1} & \bv_{2n} \end{pmatrix}.
\end{align}
Here, $\bv_{2k-1}$ and $\bv_{2k}$ are the real and imaginary parts of the conjugate pair of eigenvectors for $i\lambda_k$ and $-i\lambda_k$ (here $i$ is the imaginary number, not an index). Furthermore, we can choose $Q$ to be in $\SO(d)$. This idea of diagonalizing the signed area matrix was considered in~\cite{BS:CyclicityMultivariate}. By equivariance of the path signature, we then have
\[
    \alpha^{(2)}(Q^T\bx) = \Lambda.
\]
By the $SO(d)$-invariance of the top-level alternating signature (Theorem \ref{thm:so_invariance}), the decomposition in Equation~\ref{eq:alternating_decomp_even} can be written as
\[
    \alpha^{(2n)}(\bx) = \frac{(-1)^{n}}{(2n)! \cdot n!} \prod_{k=1}^n \lambda_k.
\]

\begin{theorem}\label{thm:vol_eigenvalues}
    Let $\bx = (x_1, \ldots, x_d): [0,1] \rightarrow \R^d$ be a cyclic curve. If $d=2n$, then
    \begin{align}
        \vol(\conv(\bx)) = \frac{(-1)^{n}}{(2n)! \cdot n!} \prod_{k=1}^n \lambda_k,
    \end{align}
    where $\lambda_k$ are the entries of $\Lambda$ in the decomposition \eqref{eq:block_diagonalize}. %are the eigenvalues from the block decomposition of $\alpha^{(2)}(\bx)$. 
    If $d=2n+1$ and the displacement $\bx(1) - \bx(0)$ is restricted to the $x_d$ coordinate ($x_k(1) - x_k(0)=0$ for $k \in [2n]$ and $x_d(1) - x_d(0) >0$), then
    \begin{align}
        \vol(\conv(\bx)) = \frac{(-1)^{n}}{(2n+1)! \cdot n!} (x_d(1) - x_d(0))\prod_{k=1}^n \lambda_k,
    \end{align}
    where $\lambda_k$ are the entries of $\Lambda$ in the decomposition \eqref{eq:block_diagonalize} of $\alpha^{(2)}(\obx)$ and $\obx = (x_1, \ldots, x_{2n}) : [0,1] \rightarrow \R^{2n}$.
\end{theorem}

With the above decomposition, we can give a geometric interpretation of the volume formula for the convex hull of the curve $\bx$. Indeed, the volume can be rewritten as sums and products of $2$-dimensional areas and $1$-dimensional lengths. More precisely, up to rotation, the volume of the convex hull of a cyclic curve is the product of the distance between the start point and end point of $\bx$, if $\bx \in \R^d$ with $d$ odd, and $\lfloor{d/2}\rfloor$ signed areas of the projections of $\bx$ onto respective $2$-planes. We spell this out for the monotone curve in the next example.
\begin{example}\label{ex:dec_moment}
    Let $\bx \subset \R^3$ be the moment curve parametrized by $(t,t^2,t^3)$. We can apply the rotation 
    \begin{equation}
        Q = \frac{1}{\sqrt{6}}
        \begin{pmatrix}
            2 & -1 & -1 \\
            0 & \sqrt{3} & -\sqrt{3} \\
            \sqrt{2} & \sqrt{2} & \sqrt{2}
        \end{pmatrix}
    \end{equation}
    which sends the vector $\bx(1)=(1,1,1)$ to $(0,0,1)$.
    Then, the off-diagonal entry $\lambda_1$ of the associated matrix $\Lambda$ is given by
    \begin{equation}
        \frac{1}{2} \int_0^1 \int_0^{t_2} \frac{2-2t_1-3t_1^2}{\sqrt{6}} \frac{(2-3t_2)t_2}{\sqrt{2}} - \frac{2-2t_2-3t_2^2}{\sqrt{6}} \frac{(2-3t_1)t_1}{\sqrt{2}} \mathrm{d} t_1 \mathrm{d} t_2 = \frac{1}{30 \sqrt{3}}.
    \end{equation}
    This value is the signed area of the projection of $\bx$ onto the plane orthogonal to $(1,1,1)$, the curve shown in Figure \ref{fig:proj_moment}.
    \begin{figure}[ht]
    \centering
      \begin{tikzpicture}
      \begin{axis}[
        width=3.2in,
        height=2.35in,
        axis lines=center,
        xlabel=\empty,
        ylabel=\empty,
        xticklabel=\empty,
        yticklabel=\empty,
        ylabel style={above right},
        xlabel style={below right},
        xtick={0},
    	ytick={0},        xmin=-0.02,xmax=0.3,ymin=-0.01,ymax=0.12]
        \addplot+[no markers, samples=200, samples y=0, domain=0:1, variable=\t, style=ultra thick, color = mygreen]
                                  ({sqrt(2/3)*t - t^2/sqrt(6)-t^3/sqrt(6)}, {t^2/sqrt(2)-t^3/sqrt(2)});
        \end{axis}
        \end{tikzpicture}
        \caption{The projection of the moment curve from Example \ref{ex:dec_moment} onto the plane orthogonal to $(1,1,1)$.}
        \label{fig:proj_moment}
    \end{figure}
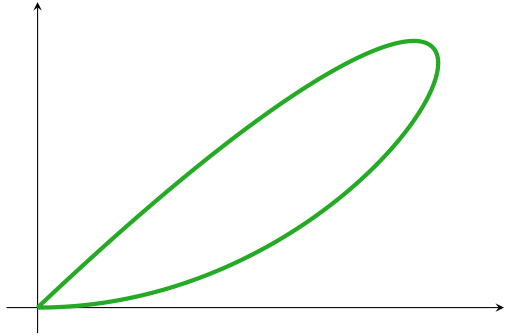
    Applying the odd version of Theorem \ref{thm:vol_eigenvalues} we get that
    $$\vol(\conv(\bx)) = \frac{1}{3!}\cdot \sqrt{3} \cdot \frac{1}{30\sqrt{3}} = \frac{1}{180}.$$
\end{example}

The interpretation above would seem to imply something incredible in odd dimension when the curve is closed, as then $x(0)=x(1)$ which gives a zero column that makes the whole determinant vanish. Hence, the formula would predict zero volume no matter the cyclic curve. This has an interesting geometric implication.

\begin{corollary}
There do not exist odd-dimensional closed cyclic curves. More generally, let $\bx\subset \R^d$ be a closed curve satisfying $\vol(\conv(\bx)) = \alpha^{(d)}(\bx)$. If $d$ is odd, then $\bx$ is contained in a hyperplane.
\end{corollary}

\section*{Acknowledgments}
We are grateful to Bernd Sturmfels for posing this problem to us, and to Anna-Laura Sattelberger for helpful discussions. We also want to thank Antonio Lerario for pointing out the connection to the Shapiro-Shapiro conjecture and to zonoid calculus. DL would like to thank Bernd Sturmfels and the MPI-MiS in Leipzig for the hospitality during a research visit where this project began. DL was supported by NCCR-Synapsy Phase-3 SNSF grant number 51NF40-185897 and Hong Kong Innovation and Technology Commission (InnoHK Project CIMDA). 

\bibliographystyle{alpha}
\bibliography{biblio}

\bigskip 

\noindent{\bf Authors' addresses:}
\smallskip
\small 

\noindent Carlos Am\'{e}ndola,
Technical University of Berlin
\hfill {\tt amendola@math.tu-berlin.de}

\noindent Darrick Lee, 
University of Oxford
\hfill {\tt darrick.lee@maths.ox.ac.uk}

\noindent Chiara Meroni,
MPI MiS Leipzig
\hfill {\tt chiara.meroni@mis.mpg.de}

\end{document}